\documentclass[preprint,11pt]{elsarticle}
\usepackage{mathrsfs}
\usepackage{amsfonts}
\usepackage{amsmath}
\usepackage{stmaryrd}
\usepackage{amssymb}
\usepackage{amsthm}
\usepackage{mathrsfs}
\usepackage{amsfonts}
\usepackage{amscd}
\usepackage{indentfirst}
\usepackage{enumerate}
\usepackage{amsmath,amsfonts,amssymb,amsthm}
\usepackage{amsmath,amssymb,amsthm,amscd}
\usepackage{graphicx,mathrsfs}
\usepackage{appendix}

\newtheorem{theorem}{Theorem}[section]
\newtheorem{proposition}[theorem]{Proposition}

\newtheorem{lemma}[theorem]{Lemma}
\newtheorem*{lemmaA}{Lemma A}
\theoremstyle{definition}

\newtheorem{definition}[theorem]{Definition}

\theoremstyle{remark}

\begin{document}
\begin{frontmatter}

\title{Existence of steady symmetric vortex patch in a disk
}


\author[1,2]{Daomin Cao}
\ead{dmcao@amt.ac.cn}

\author[3]{Guodong Wang}
\ead{wangguodong14@mails.ucas.ac.cn}

\author[1,2]{Bijun Zuo}
\ead{bjzuo@amss.ac.cn}

\address[1]{Institute of Applied Mathematics, Chinese Academy of Science, Beijing 100190, P. R. China}
\address[2]{University of Chinese Academy of Sciences, Beijing 100049, P. R. China}
\address[3]{Institute for Advanced Study in Mathematics, Harbin Institute of Technology, Harbin 150001, P. R. China}
\begin{abstract}
In this paper we construct a family of steady symmetric vortex patches for the incompressible Euler equations in a disk. The result is obtained by studying a variational problem in which the kinetic energy of the fluid is maximized subject to some appropriate constraints for the vorticity. Moreover, we show that these vortex patches shrink to a given minimum point of the corresponding Kirchhoff-Routh function as the vorticity strength parameter goes to infinity.
\end{abstract}
\end{frontmatter}



\section{Introduction}
In this paper, we prove an existence result of steady symmetric vortex patch for a planar ideal fluid in an open disk. More specifically, by maximizing the kinetic energy subject to some appropriate constraints for the vorticity, we construct such a flow in which the vorticity has the form
\begin{equation}\label{1}
  \omega^\lambda=\lambda I_{\{\psi^\lambda>\mu^\lambda\}}-\lambda I_{\{\psi^\lambda<-\mu^\lambda\}}
\end{equation}
for some $\mu^\lambda\in \mathbb{R}$. Here $I_{A}$ denotes the characteristic function of some measurable set $A$, i.e., $I_A\equiv1$ in $A$ and $I_A\equiv0$ elsewhere, $\lambda$ is the vorticity strength parameter that is given, and $\psi^\lambda$ is the stream function satisfying
\begin{equation}\label{2}
-\Delta \psi^\lambda=\omega^\lambda.
\end{equation}
In addition, $\omega^\lambda$ and $\psi^\lambda$ are both even in $x_1$ and odd in $x_2$.

In history, the construction for dynamically possible steady vortex flows has been extensively studied. Roughly speaking, there are mainly two methods dealing with this problem. The first one is called the stream-function method. It consists in finding a solution to the following semilinear elliptic equation satisfied by the stream function:
\begin{equation}\label{3}
  -\Delta \psi=f(\psi),
\end{equation}
where the nonlinearity $f$ is given. To obtain a solution for \eqref{3}, one can use the mountain pass lemma(see \cite{AS, LYY, Ni}), the constrained variational method(see \cite{Ba,BF,FB,N,SV}), or a finite dimensional reduction(see \cite{CLW,CPY}). The second one is called the vorticity method. It was put forward by Arnold(see \cite{A,AK}) and further developed by many authors; see \cite{B1,B2,B3,B4,CW,EM,EM2,T,T2} for example. The basic idea of the vorticity method is to extremize the kinetic energy of the fluid on a suitable class for the vorticity. For this method, the distributional function for the vorticity is prescribed, and the stream function $\psi$ still satisfies a semilinear elliptic equation \eqref{3}, but the nonlinearity $f$ is usually unknown. In this paper, we use the vorticity method to prove our main result.

In \eqref{1} the vorticity is a piecewise constant function, which is called a vortex patch solution. Such a special kind of solution has been studied by many authors. Here we recall some of the relevant and significant results.  In \cite{T}, by using the vorticity method, Turkington constructed a family of vortex patch solutions in a planar bounded domain. Moreover, he showed that these solutions ``shrink" to a global minimum of the Kirchhoff-Routh function with $k=1$(see Section 2 for the definition). Later in \cite{EM2}, based on a similar argument, Elcrat and Miller constructed steady multiple vortex solutions near a given strict local minimum point of the corresponding Kirchhoff-Routh function. In 2015, Cao et el \cite{CPY} proved that for any given non-degenerate critical point of the Kirchhoff-Routh function for any $k$, there exists a family of steady vortex patches shrinking to this point. The method used in \cite{CPY} is based on a finite dimensional argument for the stream function.

Notice that both in \cite{EM2} and \cite{CPY}, the non-degeneracy of the critical point for the Kirchhoff-Routh function is required. But for the simplest domain, an open disk, every critical point of the Kirchhoff-Routh function with $k\geq2$ is degenerate due to rotational invariance. A natural question arises: can we construct steady multiple vortex patches in an open disk? Our main purpose in this paper is to give this question a positive answer. The new idea here is that we add a symmetry constraint on the vorticity to ensure that the vortex patches obtained as the maximizers of the corresponding variational problem ``shrink" to two given symmetric points $P_1$ and $P_2$, even if $(P_1,P_2)$ is not an isolated minimum point of the Kirchhoff-Routh function. Then we can show that these vortex patches have the form  \eqref{1} if the vorticity strength is sufficiently large. At last by Lemma 6 in \cite{B5} these vortex patches are steady solutions of the Euler equations.

This paper is organized as follows. In Section 2, we give the mathematical formulation of the vortex patch problem and then state the main results. In Section 3, we solve a maximization problem for the vorticity and study the asymptotic behavior of the maximizers as the vorticity strength goes to infinity. In Section 4, we prove the main result. Finally in Section 5, we discuss the existence of steady non-symmetric vortex patches.

\section{Main Results}

To begin with, we introduce some notations that will be used throughout this paper. Let $D$ be the unit disk in the plane centered at the origin, that is,

\begin{equation}
  D=\{\mathbf{x}\in\mathbb{R}^2\mid|\mathbf{x}|<1\},
\end{equation}
For $\mathbf{x}=(x_1,x_2)\in D$, we will write $\bar{\mathbf{x}}=(-x_1,x_2)$ and $\tilde{\mathbf{x}}=(x_1,-x_2)$.

Let $G$ be the Green's function for $-\Delta$ in $D$ with zero
Dirichlet boundary condition, that is,
\begin{equation}\label{4}
G(\mathbf{x},\mathbf{y})=\frac{1}{2\pi}\ln \frac1{|\mathbf{x}-\mathbf{y}|}-h(\mathbf{x},\mathbf{y}), \,\,\,\mathbf{x},\mathbf{y}\in
D,
\end{equation}
where

\begin{equation}\label{5}
 h(\mathbf{x},\mathbf{y})=-\frac{1}{2\pi}\ln|\mathbf{y}|-\frac{1}{2\pi}\ln|\mathbf{x}-\frac{\mathbf{y}}{|\mathbf{y}|^2}|, \,\,\,\mathbf{x},\mathbf{y}\in
D.
\end{equation}

Let $k\geq 1$ be an integer and $\kappa_1,\kappa_2,...,\kappa_k$ be $k$ non-zero real numbers. Define the corresponding Kirchhoff-Routh function $H_k$ as follows:
\begin{equation}\label{6}
H_k(\mathbf{x_1},\cdots,\mathbf{x_k}):=-\sum_{i\neq j}\kappa_i\kappa_jG(\mathbf{x_i},\mathbf{x_j})+\sum_{i=1}^{k}\kappa_i^2h(\mathbf{x_i},\mathbf{x_i})
\end{equation}
where $(\mathbf{x}_1,\cdots,\mathbf{x}_k)\in D^{(k)}:=\underbrace{D\times D\times\cdots\times D}_{k}$ such that $\mathbf{x}_i\neq \mathbf{x}_j$ for $i\neq j$.
In this paper we consider the case $k=2$ and $\kappa_1=-\kappa_2=\kappa>0$, then the Kirchhoff-Routh function can be written as
\begin{equation}\label{7}
H_2(\mathbf{x},\mathbf{y})=2\kappa^2G(\mathbf{x},\mathbf{y})+\kappa^2h(\mathbf{x},\mathbf{x})+\kappa^2h(\mathbf{y},\mathbf{y}),
\end{equation}
where $(\mathbf{x},\mathbf{y})\in D^{(2)}$ and $\mathbf{x}\neq \mathbf{y}$. It is easy to see that
\begin{equation}
  \lim_{|\mathbf{x}-\mathbf{y}|\rightarrow 0}H_2(\mathbf{x},\mathbf{y})=+\infty, \lim_{\mathbf{x}\rightarrow\partial D\text{ or }\mathbf{y}\rightarrow\partial D}H_2(\mathbf{x},\mathbf{y})=+\infty,
\end{equation}
so $H_2$ attains its minimum in $D\times D\setminus\{(\mathbf{x},\mathbf{y})\in D\times D\mid \mathbf{x}\neq\mathbf{y}\}$. Moreover, every minimum point of $H_2$ has the form
$(\sqrt{\sqrt{5}-2}(cos\theta,sin\theta), -\sqrt{\sqrt{5}-2}(cos\theta,sin\theta))$ for some $\theta\in[0,2\pi)$, see Proposition \ref{80} in the Appendix.

Now we consider a steady ideal fluid with unit density moving in $D$ with impermeability boundary condition, which is described by the following Euler equations:
\begin{equation}\label{8}
\begin{cases}

(\mathbf{v}\cdot\nabla)\mathbf{v}=-\nabla P \,\,\,\,\,\,\,\,\,\,\,\,\,\text{in $D$},\\
\nabla\cdot\mathbf{v}=0 \,\,\,\,\,\,\,\,\,\,\,\,\,\,\,\,\,\,\,\,\,\,\,\,\,\,\,\,\,\,\,\,\,\text{in $D$},\\
\mathbf{v}\cdot {\mathbf{n}}=0 \,\,\,\,\,\,\,\,\,\,\,\,\,\,\,\,\,\,\,\,\,\,\,\,\,\,\,\,\,\,\,\,\,\,\text{on $\partial D$ },
\end{cases}
\end{equation}
where $\mathbf{v}=(v_1,v_2)$ is the velocity field, $P$ is the scalar pressure, and ${\mathbf{n}}(\mathbf{x})$ is the outward unit normal at $\mathbf{x}\in\partial D$.

To simplify The Euler equations \eqref{8}, we define the vorticity $\omega :=\partial_1 v_2-\partial_2v_1$. By using the identity $\frac{1}{2}\nabla|\mathbf{v}|^2=(\mathbf{v}\cdot\nabla)\mathbf{v}+\mathbf{v}^\perp\omega$, where $J(v_1,v_2)=(v_2,-v_1)$ denotes clockwise rotation through $\frac{\pi}{2}$,
 the first equation of $\eqref{1}$ can be written as
\begin{equation}\label{9}
 \nabla(\frac{1}{2}|\mathbf{v}|^2+P)-\mathbf{v}^\perp\omega=0.
\end{equation}
Taking the curl in $\eqref{9}$ we get the $vorticity$ $equation$

\begin{equation}\label{10}
\nabla\cdot(\omega\mathbf{v})=0.
\end{equation}

To recover the velocity field from the vorticity, we define the stream function $\psi$ by solving the following Poisson's equation with zero
Dirichlet boundary condition
\begin{equation}\label{11}
 \begin{cases}
-\Delta \psi=\omega\text{\quad in $D$,}\\
\psi=0 \text{\quad\quad\,\, on $\partial D$.}
\end{cases}
\end{equation}
Using the Green's function we have

\begin{equation}
 \psi(\mathbf{x})=G\omega(\mathbf{x}):=\int_DG(\mathbf{x},\mathbf{y})\omega(\mathbf{y})d\mathbf{y}.
\end{equation}
Since $D$ is simply-connected and $\mathbf{v}\cdot\mathbf{n}$ on $\partial D$, it is easy to check that $\mathbf{v}$ can be uniquely determined by $\omega$ in the following way(see \cite{MPu}, Chapter 1, Theorem 2.2 for example)
\begin{equation}\label{12}
\mathbf{v}=\nabla^\perp\psi,
\end{equation}
where $\nabla^\perp\psi:=(\nabla\psi)^\perp=(\partial_2\psi,-\partial_1\psi)$.

From the above formal discussion we get the following vorticity form of the Euler equations:

\begin{equation}\label{13}
\begin{cases}

\nabla\cdot(\omega\mathbf{v})=0,\\
\mathbf{v}=\nabla^\perp G\omega.
\end{cases}
\end{equation}

In this paper, we interpret the vorticity equation \eqref{13} in the weak sense.

\begin{definition}\label{14}
We call $\omega\in L^\infty(D)$ a weak solution of \eqref{13} if

\begin{equation}\label{15}
\int_D\omega\nabla^\perp G\omega\cdot\xi dx=0
\end{equation}
for all $\xi\in C_0^{\infty}(D)$.
\end{definition}

It should be noted that if $\omega\in L^{\infty}(D)$, then by the regularity theory for elliptic equations $\psi\in C^{1,\alpha}(D)$ for some $\alpha\in(0,1)$, so \eqref{15} makes sense for all $\omega\in L^{\infty}(D)$.

The following lemma from \cite{B5} gives a criterion for weak solutions of \eqref{13}.

\begin{lemmaA}\label{A}
Let $\omega\in L^\infty(D)$ and $\psi=G\omega$. Suppose that $\omega=f(\psi)$ a.e. in $D$ for some monotonic function $f:\mathbb{R}\rightarrow\mathbb{R}$, then $\omega$ is a weak solution of \eqref{13}.
\end{lemmaA}

Now we can state our main result.

\begin{theorem}\label{144}
Let $\kappa$ be a positive number. Then there exists a $\lambda_0>0$ such that for any $\lambda>\lambda_0$, there exists $\omega^\lambda\in L^\infty(D)$ satisfying
\begin{enumerate}
\item $\omega^\lambda$ is a weak solution of \eqref{13};
\item $\omega^\lambda=\omega^\lambda_1+\omega^\lambda_2,$ where $\omega^\lambda_1=\lambda I_{\{\psi^\lambda>\mu^\lambda\}}$ and $\omega^\lambda_2=-\lambda I_{\{\psi^\lambda<-\mu^\lambda\}}$ for some $\mu^\lambda\in\mathbb{R}^+$ depending on $\lambda$, and $\psi^\lambda=G\omega^\lambda;$
\item $\int_D\omega^\lambda_1(\mathbf{x})d\mathbf{x}=\kappa,$ $\int_D\omega^\lambda_2(\mathbf{x})d\mathbf{x}=-\kappa;$
\item $\omega^\lambda$ is even in $x_1$ and odd in $x_2$, that is, $\omega^\lambda(\mathbf{x})=\omega^\lambda(\bar{\mathbf{{x}}})$ and $\omega^\lambda(\mathbf{x})=-\omega^\lambda(\tilde{\mathbf{{x}}});$

\item $diam(supp(\omega^\lambda_1))=diam(supp(\omega_2^\lambda))\leq C\lambda^{-\frac{1}{2}}$, where $C$ is a positive number not depending on $\lambda;$
\item $\lim_{\lambda\rightarrow+\infty}|\frac{1}{\kappa}\int_D\mathbf{x}\omega^\lambda_1(\mathbf{x})d\mathbf{x}-P_1|=0$, $\lim_{\lambda\rightarrow+\infty}|-\frac{1}{\kappa}\int_D\mathbf{x}\omega^\lambda_2(\mathbf{x})d\mathbf{x}-P_2|=0$, where $P_1=(0,\sqrt{\sqrt{5}-2})$ and $P_2=(0,-\sqrt{\sqrt{5}-2})$.
\end{enumerate}

\end{theorem}

\section{Variational Problem}
In this section we study a maximization problem for the vorticity and give some estimates for the maximizers as the vorticity strength goes to infinity.

Firstly we choose $\delta>0$ sufficiently small such that $B_{\delta}(P_i)\subset\subset D$ for $ i=1,2$ and $\overline{B_{\delta}(P_1)}\cap \overline{B_{\delta}(P_2)}=\varnothing$. For example, we can choose $\delta=\frac{\sqrt{\sqrt{5}-2}}{2}$.
In the rest of this paper, we use $B_i$ to denote $B_\delta(P_i)$ for $i=1,2$ for simplicity.

For $\lambda>0$ sufficiently large, we define the vorticity class $K^\lambda$ as follows:
\begin{equation}\label{31}
\begin{split}
K^\lambda:=\big{\{}\omega\in L^\infty(D)\,\mid\,\,\omega=\omega_1+\omega_2, supp(\omega_i)\subset B_i\text{\,\,for\,\,}i=1,2, \int_D\omega_1(\mathbf{x})d\mathbf{x}=\kappa,
\\0\leq \omega_1\leq\lambda,  \omega_1(\mathbf{x})=\omega_1(\bar{\mathbf{x}})\text{\,\,and\,\,} \omega_2(\mathbf{x})=\omega_1(\tilde{\mathbf{x}})\text{\,\,for\,\,} \mathbf{x}\in D
 \big{\}}.
\end{split}
\end{equation}

It is easy to check that for any $\omega\in K^\lambda$, $\omega$ is even in $x_1$ and odd in $x_2$. It is also clear that $K_\lambda$ is not empty if $\lambda>0$ is large enough.

The kinetic energy of the fluid with vorticity $\omega$ is
\begin{equation}\label{32}
E(\omega)=\frac{1}{2}\int_D\int_DG(\mathbf{x},\mathbf{y})\omega(\mathbf{x})\omega(\mathbf{y})d\mathbf{x}d\mathbf{y},\,\,\omega\in K^\lambda.
\end{equation}
Integrating by parts we also have
 \begin{equation}
E(\omega)=\frac{1}{2}\int_D\psi(\mathbf{x})\omega(\mathbf{x})d\mathbf{x}=\frac{1}{2}\int_D|\nabla\psi(\mathbf{x})|^2d\mathbf{x}=\frac{1}{2}\int_D|\mathbf{v}(\mathbf{x})|^2d\mathbf{x}.
\end{equation}

In the rest of this section we consider the maximization of $E$ on $K_\lambda$ and study the properties of the maximizer.
\subsection{Existence of a maximizer}
First we show the existence of a maximizer of $E$ on $K^\lambda$.
\begin{lemma}\label{33}
There exists $\omega^\lambda\in K_\lambda$ such that $E(\omega^\lambda)=\sup_{\omega\in K_\lambda}E(\omega)$.
\end{lemma}
\begin{proof}
The proof can be divided into three steps.

\textbf{Step 1}: $E$ is bounded from above on $K^\lambda$. In fact, since $|\omega|_{L^\infty(D)}\leq \lambda$ for any $\omega\in K^\lambda$, we have
\[E(\omega)=\frac{1}{2}\int_D\int_DG(\mathbf{x},\mathbf{y})\omega(\mathbf{x})\omega(\mathbf{y})d\mathbf{x}d\mathbf{y}\leq \frac{1}{2}\lambda^2\int_D\int_D|G(\mathbf{x},\mathbf{y})|d\mathbf{x}d\mathbf{y}\leq C\lambda^2
\]
for some absolute constant $C$. Here we use the fact that $G\in L^1(D\times D)$. This gives
\[\sup_{\omega\in K_\lambda}E(\omega)<+\infty.\]

\textbf{Step 2}: $K^\lambda$ is closed in the weak$*$ topology of $L^\infty(D)$, or equivalently, for any sequence $\{\omega_n\}\subset K_\lambda$ and $\omega\in L^\infty(D)$ satisfying

\begin{equation}\label{34}
\lim_{n\rightarrow+\infty} \int_D\omega_n(\mathbf{x})\phi(\mathbf{x}) d\mathbf{x}=\int_D\omega(\mathbf{x})\phi(\mathbf{x}) d\mathbf{x},\,\,\forall\, \phi\in L^1(D),
\end{equation}
we have $\omega\in K^\lambda$.

To prove this, we first show that
\begin{equation}\label{35}
supp(\omega)\subset\cup_{i=1}^2B_i.
\end{equation}
In fact, for any $\phi\in C_0^\infty(D\setminus \cup_{i=1}^2 \overline{B_i})$, by \eqref{34} we have
\[
\int_D\omega(\mathbf{x})\phi(\mathbf{x})d\mathbf{x}=\lim_{n\rightarrow +\infty}\int_D\omega_n(\mathbf{x})\phi(\mathbf{x})d\mathbf{x}=0,
\]
which means that $\omega=0$ a.e. in $D\setminus \cup_{i=1}^2 \overline{B_i}$.

Now we define $\omega_i=\omega I_{B_i}$ for $i=1,2$. It is obvious that $\omega=\omega_1+\omega_2$. By choosing $\phi=I_{B_1}$ in \eqref{34}, we have
\begin{equation}\label{36}
\int_D\omega_1(\mathbf{x})d\mathbf{x}=\int_D\omega(\mathbf{x})\phi(\mathbf{x})d\mathbf{x}=\lim_{n\rightarrow +\infty}\int_D\omega_n(\mathbf{x})\phi(\mathbf{x})d\mathbf{x}=\lim_{n\rightarrow +\infty}\int_{B_1}\omega_n(\mathbf{x})d\mathbf{x}=\kappa.
\end{equation}

It is also easy to show that $0\leq \omega_1\leq\lambda$ in $ D$. In fact, suppose that $|\{\omega_1>\lambda\}|>0$, then we can choose $\varepsilon_0, \varepsilon_1>0$ such that $|\{\omega_1>\lambda+\varepsilon_0\}|>\varepsilon_1$. Denote $S=\{\omega_1>\lambda+\varepsilon_0\}\subset B_1$, then by choosing $\phi=I_{S}$ in \eqref{34} we have
\[\int_{S}(\omega_1-\omega_n)(\mathbf{x})d\mathbf{x}\geq\varepsilon_0|S|\geq\varepsilon_0\varepsilon_1.\]
On the other hand, by \eqref{34}
\[\lim_{n\rightarrow +\infty}\int_{S}(\omega_1-\omega_n)(\mathbf{x})d\mathbf{x}=\lim_{n\rightarrow +\infty}\int_D(\omega-\omega_n)(\mathbf{x})\phi(\mathbf{x})d\mathbf{x}=0,\]
which is a contradiction. So we have $\omega_1\leq \lambda$. Similarly we can prove $\omega_1\geq 0$.

To finish Step 2, it suffices to show that $\omega$ is even in $x_1$ and odd in $x_2$. For fixed $\mathbf{x}\in D\cap\{x_1>0\}$, define $\phi=\frac{1}{\pi s^2}I_{B_s(\mathbf{x})}-\frac{1}{\pi s^2}I_{B_s(\bar{\mathbf{x}})}$, where $s>0$ is sufficiently small. Since $\omega_n$ is even in $x_1$ for each $n$ and $\phi$ is odd in $x_1$, by \eqref{34} we have

\begin{equation}\label{37}
\int_D\omega(\mathbf{y})\phi(\mathbf{y}) d\mathbf{y}=\lim_{n\rightarrow+\infty} \int_D\omega_n(\mathbf{y})\phi(\mathbf{y})d\mathbf{y}=0,
\end{equation}
which means that

\begin{equation}\label{38}
\frac{1}{|B_s(\mathbf{x})|}\int_{B_s(\mathbf{x})}\omega(\mathbf{y}) d\mathbf{y}=\frac{1}{|B_s(\bar{\mathbf{x}})|}\int_{B_s(\bar{\mathbf{x}})}\omega(\mathbf{y}) d\mathbf{y}.
\end{equation}
By Lebesgue differential theorem, for a.e. $\mathbf{x}\in D\cap\{x_1>0\}$, we have
\begin{equation}\label{39}
\lim_{s\rightarrow0^+}\frac{1}{|B_s(\mathbf{x})|}\int_{B_s(\mathbf{x})}\omega(\mathbf{y}) d\mathbf{y}=\omega(\mathbf{x}),
\end{equation}
and
\begin{equation}\label{310}
\lim_{s\rightarrow0^+}\frac{1}{|B_s(\bar{\mathbf{x}})|}\int_{B_s(\bar{\mathbf{x}})}\omega(\mathbf{y}) d\mathbf{y}=\omega(\bar{\mathbf{x}}).
\end{equation}

Combining \eqref{38}, \eqref{39} and \eqref{310} we get
\begin{equation}\label{311}
  \omega(\mathbf{x})=\omega(\bar{\mathbf{x}})\text{\,\,\,a.e.\,\,} x\in D\cap\{x_1>0\},
\end{equation}
which means that $\omega$ is even in $x_1$. Similarly we can prove that $\omega$ is odd in $x_2$.

From all the above arguments we know that $\omega\in K^\lambda$.

\textbf{Step 3}: $E$ is weak$*$ continuous on $K^\lambda$, that is, for any sequence $\{\omega_n\}\subset K^\lambda$ such that $\omega_n\rightarrow\omega$ weakly$^*$ in $L^\infty(D)$ as $n\rightarrow+\infty$, we have
\[\lim_{n\rightarrow+\infty}E(\omega_n)= E(\omega).\]
In fact, by \eqref{34} as $n\rightarrow+\infty$ we know that $\omega_n\rightarrow\omega$ weakly in $L^2(D)$, then $\psi_n\rightarrow \psi$ weakly in $W^{2,2}(D)$ thus strongly in $L^2(D)$, where $\psi_n=G\omega_n$ and $\psi=G\omega$. So we get as $n\rightarrow+\infty$

\[E(\omega_n)=\frac{1}{2}\int_D\psi_n(\mathbf{x})\omega_n(\mathbf{x})d\mathbf{x}\rightarrow\frac{1}{2}\int_D\psi(\mathbf{x})\omega(\mathbf{x})d\mathbf{x}=E(\omega).\]

Now we finish the proof of Lemma \ref{33} by using the standard maximization technique. By Step 1, we can choose a maximizing sequence $\{\omega_n\}\subset K^\lambda$ such that

\[\lim_{n\rightarrow+\infty}E(\omega_n)=\sup_{\omega\in K^\lambda}E(\omega)\]
Since $K^\lambda$ is bounded thus weakly$*$ compact in $L^\infty(D)$, we can choose a subsequence $\{\omega_{n_k}\}$ such that as $k\rightarrow+\infty$ $\omega_{n_k}\rightarrow\omega^\lambda$ weakly$*$ in $L^\infty(D)$ for some $\omega^\lambda\in L^\infty(D)$. By Step 2, we have $\omega^\lambda\in K^\lambda.$ Finally by Step 3 we have
\[E(\omega^\lambda)=\lim_{k\rightarrow+\infty}E(\omega_{n_k})=\sup_{\omega\in K^\lambda}E(\omega),\]
which is the desired result.

\end{proof}

\subsection{Profile of $\omega^\lambda$}
Since $\omega^\lambda\in K^\lambda$, we know that $\omega^\lambda$ has the form $\omega^\lambda=\omega^\lambda_1+\omega^\lambda_2$ with $\omega^\lambda_1$ and $\omega^\lambda_2$ satisfying
\begin{enumerate}
\item $supp(\omega^\lambda_i)\subset B_i$ for $i=1,2$,
\item $\int_D\omega^\lambda_1(\mathbf{x})d\mathbf{x}=-\int_D\omega^\lambda_2(\mathbf{x})d\mathbf{x}=\kappa$,
\item $\omega^\lambda_1(\mathbf{x})=\omega^\lambda_1(\bar{\mathbf{x}})$, $\omega^\lambda_1(\mathbf{x})=-\omega^\lambda_2(\tilde{\mathbf{x}})$ for any $\mathbf{x}\in D$.
\end{enumerate}
In fact, we can prove that $\omega^\lambda$ has a special form.
\begin{lemma}\label{312}
There exists $\mu^\lambda\in \mathbb{R}$ depending on $\lambda$ such that
\[\omega^\lambda_1=\lambda I_{\{\psi^\lambda>\mu^\lambda\}\cap B_1}\]
and
\[\omega^\lambda_2=-\lambda I_{\{\psi^\lambda<-\mu^\lambda\}\cap B_2}\]
where $\psi^\lambda:= G\omega^\lambda$.
\end{lemma}

\begin{proof}
First we show that $\omega^\lambda_1$ has the form $\omega^\lambda_1=\lambda I_{\{\psi^\lambda>\mu^\lambda\}\cap B_1}.$ Choose $\alpha,\beta\in L^\infty(D)$ satisfying
\begin{equation}\label{313}
\begin{cases}
\alpha, \beta\geq 0,\,\, \int_D\alpha(\mathbf{x})d\mathbf{x}=\int_D \beta({\mathbf{x}})d\mathbf{x},
 \\supp(\alpha),supp(\beta)\subset B_1\cap\{x_1>0\},
 \\ \alpha=0\text{\,\,\,\,\,\,} \text{in}\text{\,\,} D\setminus\{\omega^\lambda_1\leq\lambda-a\},
 \\ \beta=0\text{\,\,\,\,\,\,} \text{in}\text{\,\,} D\setminus\{\omega^\lambda_1\geq a\},
\end{cases}
\end{equation}
where $a$ is a small positive number.
Now we define a family of test functions $\omega_s=\omega^\lambda+s(g_1-g_2)$ for $s>0$ sufficiently small, where
\[g_1(x)=\alpha(\mathbf{x})+\alpha(\bar{\mathbf{x}})-\alpha(\tilde{\mathbf{x}})-\alpha(\tilde{\bar{\mathbf{x}}})\]
and
\[g_2(x)=\beta(\mathbf{x})+\beta(\bar{\mathbf{x}})-\beta(\tilde{\mathbf{x}})-\beta(\tilde{\bar{\mathbf{x}}}).\]

Note that $g_1, g_2$ are both even in $x_1$ and odd in $x_2$, and $supp(g_1), supp(g_2)\subset B_1\cup B_2$.

It is easy check that $\omega_s\in K^\lambda$ if $s$ is sufficiently small(depending on $\alpha, \beta$ and $a$). So we have

\begin{equation}\label{314}
0\geq\frac{dE(\omega_{s})}{ds}\bigg{|}_{s=0^+}=\int_Dg_1(\mathbf{x})\psi^\lambda(\mathbf{x})d\mathbf{x}-\int_Dg_2(\mathbf{x})\psi^\lambda(\mathbf{x})d\mathbf{x}.
\end{equation}

On the other hand ,
\begin{equation}\label{315}
  \begin{split}
   &\int_Dg_1(\mathbf{x})\psi^\lambda(\mathbf{x})d\mathbf{x}-\int_Dg_2(\mathbf{x})\psi^\lambda(\mathbf{x})d\mathbf{x}
      =\int_{D_1}\psi^\lambda(\mathbf{x})(g_1-g_2)(\mathbf{x})d\mathbf{x}\\
      &+\int_{D_2}\psi^\lambda(\mathbf{x})(g_1-g_2)(\mathbf{x})d\mathbf{x} +\int_{D_3}\psi^\lambda(\mathbf{x})(g_1-g_2)(\mathbf{x})d\mathbf{x}+\int_{D_4}\psi^\lambda(\mathbf{x})(g_1-g_2)(\mathbf{x})d\mathbf{x},
  \end{split}
\end{equation}
where \[D_1=B_1\cap\{x_1>0\}, D_2=B_1\cap\{x_1<0\},\]
 \[ D_3=B_2\cap\{x_1<0\}, D_4=B_2\cap\{x_1>0\}.\]
Since $\psi^\lambda$ is also even in $x_1$ and odd in $x_2$(this can be proved directly using the symmetry of the Green's function), we have
\begin{equation}\label{316}
\begin{split}
  &\int_{D_1}\psi^\lambda(\mathbf{x})(g_1-g_2)(\mathbf{x})d\mathbf{x}=\int_{D_2}\psi^\lambda(\mathbf{x})(g_1-g_2)(\mathbf{x})d\mathbf{x} \\
  =&\int_{D_3}\psi^\lambda(\mathbf{x})(g_1-g_2)(\mathbf{x})d\mathbf{x}=\int_{D_4}\psi^\lambda(\mathbf{x})(g_1-g_2)(\mathbf{x})d\mathbf{x}.
  \end{split}
\end{equation}

Combining \eqref{314}, \eqref{315} and \eqref{316} we conclude that
\begin{equation}\label{317}
\int_{D_1}\psi^\lambda(\mathbf{x})g_1(\mathbf{x})d\mathbf{x}\leq \int_{D_1}\psi^\lambda(\mathbf{x})g_2(\mathbf{x})d\mathbf{x},
\end{equation}
or equivalently
\begin{equation}\label{318}
\int_{D_1}\psi^\lambda(\mathbf{x})\alpha(\mathbf{x})d\mathbf{x}\leq \int_{D_1}\psi^\lambda(\mathbf{x})\beta(\mathbf{x})d\mathbf{x}.
\end{equation}

By the choice of $\alpha$ and $\beta$, inequality \eqref{317} holds if and only if

\begin{equation}\label{319}
\sup_{\{\omega^\lambda<\lambda\}\cap D_1}\psi^\lambda\leq \inf_{\{\omega^\lambda>0\}\cap D_1}\psi^\lambda.
\end{equation}
Combining the continuity of $\psi^\lambda$ in $\{\omega^\lambda<\lambda\}\cap D_1$, we obtain
\begin{equation}\label{320}
\sup_{\{\omega^\lambda<\lambda\}\cap D_1}\psi^\lambda= \inf_{\{\omega^\lambda>0\}\cap D_1}\psi^\lambda.
\end{equation}

Now we define $\mu^\lambda$ as follows

\begin{equation}\label{321}
\mu^\lambda:=\sup_{\{\omega^\lambda<\lambda\}\cap D_1}\psi^\lambda= \inf_{\{\omega^\lambda>0\}\cap D_1}\psi^\lambda.
\end{equation}
It is easy to see that
\begin{equation}
\begin{cases}
\omega^\lambda=0\text{\,\,\,\,\,\,$$\,} \text{a.e. in}\text{\,\,}\{\psi^\lambda<\mu^\lambda\}\cap D_1,
 \\ \omega^\lambda=\lambda\text{\,\,\,\,\,\,$$\,} \text{a.e. in}\text{\,\,}\{\psi^\lambda>\mu^\lambda\}\cap D_1.
\end{cases}
\end{equation}
On $\{\psi^\lambda_1=\mu^\lambda\}\cap D_1$, $\psi^\lambda$ is a constant, so we have $\nabla\psi^\lambda=0$,  then $\omega^\lambda=-\Delta \psi^\lambda =0$.
Hence we conclude that
\begin{equation}
\begin{cases}
\omega^\lambda=0\text{\,\,\,\,\,\,$$\,} \text{a.e. in}\text{\,\,}\{\psi^\lambda\leq\mu^\lambda\}\cap D_1,
 \\ \omega^\lambda=\lambda\text{\,\,\,\,\,\,$$\,} \text{a.e. in}\text{\,\,}\{\psi^\lambda>\mu^\lambda\}\cap D_1.
\end{cases}
\end{equation}

Finally by the symmetry of $\omega^\lambda$ and $\psi^\lambda$ we get the desired result.

\end{proof}

\subsection{Asymptotic behavior of $\omega^\lambda$ as $\lambda\rightarrow+\infty$}
Now we give some asymptotic estimates on $\omega^\lambda$ as $\lambda\rightarrow +\infty$. We will use $C$ to denote various various numbers not depending on $\lambda$.

\begin{lemma}\label{322}
Let $\varepsilon=\sqrt{\frac{\kappa}{\lambda\pi}}$. Suppose that $\omega^\lambda$ is obtained as in Lemma \ref{33}. Then
\begin{enumerate}
\item $E(\omega^\lambda)\geq-\frac{1}{2\pi}\kappa^2\ln\varepsilon-C$;
\item $\mu^\lambda\geq-\frac{1}{2\pi}\kappa\ln\varepsilon-C;$
\item there exists a positive number $R>1$, not depending on $\lambda$, such that $diam(supp(\omega^\lambda_i))<R\varepsilon$ for $i=1,2$;
\item $\frac{1}{\kappa}\int_D\mathbf{x}\omega^\lambda_1(\mathbf{x})d\mathbf{x}\rightarrow \mathbf{x}_1$, $-\frac{1}{\kappa}\int_D\mathbf{x}\omega^\lambda_2(\mathbf{x})d\mathbf{x}\rightarrow \mathbf{x}_2$, as $\lambda \rightarrow +\infty$, where $\mathbf{x}_1,\mathbf{x}_2$ satisfies
    $\mathbf{x}_1\in\overline{B_1},\mathbf{x}_2\in\overline{B_2}$ and $H_2(\mathbf{x}_1,\mathbf{x}_2)=\min_{(\mathbf{x},\mathbf{y})\in D\times D}H_2(\mathbf{x},\mathbf{y})$.

\end{enumerate}
\end{lemma}
\begin{proof}
The proof is exactly the same as the one in \cite{CW}, we omit it here therefore.
\end{proof}

\section{Proof for Theorem \ref{144}}
In this section we prove Theorem \ref{144}. The key point is to show that the maximizer $\omega^\lambda$ obtained in Lemma \ref{33} satisfies the condition in Lemma A.

\begin{proof}[Proof of Theorem \ref{144}]
First, we show that the support of $\omega^\lambda_i$ shrinks to $P_i$ for $i=1,2$. By (3) and (4) of Lemma \ref{322}, $supp(\omega^\lambda_i)$ shrinks to $\mathbf{x}_i$, where $(\mathbf{x}_1,\mathbf{x}_2)$ is a minimum point of $H_2$ in $D\times D$. Then by Proposition \ref{80} in the Appendix, there exists some $\theta\in[0,2\pi)$ such that
\[\mathbf{x}_1=\sqrt{\sqrt{5}-2}(cos\theta,sin\theta), \mathbf{x}_2=-\sqrt{\sqrt{5}-2}(cos\theta,sin\theta).\]
But by the symmetry of $\omega^\lambda$(see the definition of $K^\lambda$), we have $cos\theta=0$, $sin\theta>0$, so $\theta=\frac{\pi}{2}$. That is,
\begin{equation}\label{1001}
\mathbf{x}_1=(0,\sqrt{\sqrt{5}-2}), \mathbf{x}_2=(0,-\sqrt{\sqrt{5}-2}).
\end{equation}

Second, we show that $\omega^\lambda$ is a weak solution of \eqref{13}. To begin with, we show that
\begin{equation}\label{1002}
|\psi^\lambda|\leq C \text{\,\,on\,\,}\partial{B_1}\cup\partial{B_2},
\end{equation}
where $C$ is a positive number not depending on $\lambda$.
In fact, by $(3)(4)$ of Lemma \ref{322} and \eqref{1001}, we have
\begin{equation}\label{1003}
  dist(supp\omega^\lambda_1,\partial{B_i})>\delta_0, i=1,2,
\end{equation}
for some $\delta_0\in(0,1)$ not depending on $\lambda$ if $\lambda$ is sufficiently large. Then for $\mathbf{x}\in\partial B_1$,
\begin{equation}\label{1004}
\begin{split}
  |\psi^\lambda(\mathbf{x})|&=|\int_DG(\mathbf{x},\mathbf{y})\omega^\lambda(\mathbf{y})d\mathbf{y}|\\
  &=|\int_DG(\mathbf{x},\mathbf{y})\omega_1^\lambda(\mathbf{y})d\mathbf{y}+\int_DG(\mathbf{x},\mathbf{y})\omega_2^\lambda(\mathbf{y})d\mathbf{y}|\\
  &=|\int_D-\frac{1}{2\pi}\ln|\mathbf{x}-\mathbf{y}|\omega_1^\lambda(\mathbf{y})d\mathbf{y}-\int_Dh(\mathbf{x},\mathbf{y})\omega_1^\lambda(\mathbf{y})d\mathbf{y}
  +\int_DG(\mathbf{x},\mathbf{y})\omega_2^\lambda(\mathbf{y})d\mathbf{y}|\\
  &\leq |-\frac{1}{2\pi}\ln\delta_0\int_D\omega_1^\lambda(\mathbf{y})d\mathbf{y}|+\int_{supp(\omega^\lambda_1)}|h(\mathbf{x},\mathbf{y})|\omega_1^\lambda(\mathbf{y})d\mathbf{y}
  +\int_{supp(\omega^\lambda_2)}|G(\mathbf{x},\mathbf{y})\omega_2^\lambda(\mathbf{y})|d\mathbf{y}\\
  &\leq -\frac{\kappa}{2\pi}\ln\delta_0+C.
  \end{split}
\end{equation}
Here we used the fact that $|h|\leq C$ on $\partial B_1\times supp(\omega^\lambda_1)$ and $|G|\leq C$ on $\partial B_1\times supp(\omega^\lambda_2)$. Similarly $|\psi^\lambda|\leq C$ on $\partial B_2$. So \eqref{1002} is proved.

 By Lemma \ref{312}, $\omega^\lambda$ has the form

\begin{equation}\label{41}
\omega^\lambda=\lambda I_{\{\psi^\lambda>\mu^\lambda\}\cap B_1}-\lambda I_{\{\psi^\lambda<-\mu^\lambda\}\cap B_2}.
\end{equation}

Since $|\psi^\lambda|\leq C$ on $\partial B_1\cup \partial B_2$, by the maximum principle we know that
\begin{equation}\label{42}
  \psi^\lambda\leq C \text{\,\,in\,\,} D\setminus B_1, \psi^\lambda\geq -C \text{\,\,in\,\,} D\setminus B_2.
\end{equation}

On the other hand, by $(2)$ of Lemma \ref{322} we have $\lim_{\lambda\rightarrow+\infty}\mu^\lambda=+\infty$, so combining \eqref{42} we get
\begin{equation}\label{43}
 \{\psi^\lambda>\mu^\lambda\}\cap B_1=\{\psi^\lambda>\mu^\lambda\}, \{\psi^\lambda<-\mu^\lambda\}\cap B_2=\{\psi^\lambda<-\mu^\lambda\}
\end{equation}
provided $\lambda$ is sufficiently large.
So in fact $\omega^\lambda$ has the form

\begin{equation}\label{44}
\omega^\lambda=\lambda I_{\{\psi^\lambda>\mu^\lambda\}}-\lambda I_{\{\psi^\lambda<-\mu^\lambda\}},
\end{equation}
or equivalently,
\begin{equation}\label{45}
\omega^\lambda=f(\psi^\lambda),
\end{equation}
where $f:\mathbb{R}\rightarrow\mathbb{R}$ is an non-decreasing function defined by
\begin{equation}\label{46}
f(t)=\left\{
\begin{aligned}
&\lambda, \quad  t>\mu^\lambda,\\
&0,  \quad  t\in[-\mu^\lambda,\mu^\lambda],\\
-&\lambda,\quad  t<-\mu^\lambda.
\end{aligned}
\right.
\end{equation}
Then by Lemma A, $\omega^\lambda$ is a weak solution of \eqref{13}.

Last, combining the other properties of $\omega^\lambda$ obtained in Section 2 we finish the proof of Theorem \ref{144}.
\end{proof}

\section{Non-symmetric Case}
In this section, we briefly discuss the existence of steady non-symmetric vortex patch.
\begin{theorem}\label{51}
Let $\kappa_1,\kappa_2$ be two real numbers such that $\kappa_1>0,\kappa_2<0$. Then there exists a $\lambda_0>0$ such that for any $\lambda>\lambda_0$, there exists $\omega^\lambda\in L^\infty(D)$ satisfying
\begin{enumerate}
\item $\omega^\lambda=\omega^\lambda_1+\omega^\lambda_2,$ where $\omega^\lambda_1=\lambda I_{\{\psi^\lambda>\mu^\lambda_1\}}$ and $\omega^\lambda_2=-\lambda I_{\{\psi^\lambda<-\mu^\lambda_2\}}$ for some $\mu^\lambda_1,\mu^\lambda_2\in\mathbb{R}^+$ depending on $\lambda$, and $\psi^\lambda=G\omega^\lambda;$
\item $\int_D\omega^\lambda_1(\mathbf{x})d\mathbf{x}=\kappa_1,$ $\int_D\omega^\lambda_2(\mathbf{x})d\mathbf{x}=\kappa_2;$
\item $\omega^\lambda$ is even in $x_1$, that is, $\omega^\lambda(\mathbf{x})=\omega^\lambda(\bar{\mathbf{{x}}})$;

\item $diam(supp(\omega^\lambda_1)), diam(supp(\omega_2^\lambda))\leq C\lambda^{-\frac{1}{2}}$, where $C$ is a positive number not depending on $\lambda;$
\item $\lim_{\lambda\rightarrow+\infty}|\frac{1}{\kappa_1}\int_D\mathbf{x}\omega^\lambda_1(\mathbf{x})d\mathbf{x}-P|=0$, $\lim_{\lambda\rightarrow+\infty}|\frac{1}{\kappa_2}\int_D\mathbf{x}\omega^\lambda_2(\mathbf{x})d\mathbf{x}-Q|=0$, where $P=(0,p),Q=(0,q)$ are two points in $D$ with $p>0,q<0$ depending only on $\frac{\kappa_2}{\kappa_1}$.
\end{enumerate}

\end{theorem}

\begin{proof}
The construction of $\omega^\lambda$ here is similar to the one of the symmetric case in Section 2. For simplicity we only sketch the proof.

First we choose $(P,Q)$ as a minimum point of the corresponding Kirchhoff-Routh function

\begin{equation}\label{52}
H_2(\mathbf{x},\mathbf{y}):=-2\kappa_1\kappa_2G(\mathbf{x},\mathbf{y})+{\kappa_1}^2h(\mathbf{x},\mathbf{x})+{\kappa_2}^2h(\mathbf{y},\mathbf{y}),\,\,\mathbf{x},\mathbf{y}\in D, \mathbf{x}\neq\mathbf{y}.
\end{equation}
Let $P=(0,p),Q=(0,q)$ with $p>0,q<0$. In this case, $P,Q$ are uniquely determined by $\frac{\kappa_2}{\kappa_1}$(see Proposition \ref{80} in the Appendix). Now we choose $\delta>0$ sufficiently small such that $B_{\delta}(P),B_{\delta}(Q)\subset\subset D$ and $\overline{B_{\delta}(P)}\cap \overline{B_{\delta}(Q)}=\varnothing$.

Consider the maximization of $E$ on the following class

\begin{equation}\label{53}
\begin{split}
M^\lambda:=\big{\{}\omega\in L^\infty(D)\,\mid\,\,\omega=\omega_1+\omega_2, supp(\omega_1)\subset B_\delta(P), supp(\omega_2)\subset B_\delta(Q),\\ \int_D\omega_i(\mathbf{x})d\mathbf{x}=\kappa_i,
0\leq sgn(\kappa_i)\omega_i\leq\lambda, \text{\,\,for\,\,}i=1,2, \,\omega(\mathbf{x})=\omega(\bar{\mathbf{x}})\text{\,\,for\,\,} x\in D
 \big{\}}.
\end{split}
\end{equation}

Then by repeating the procedures in Section 2, we can prove that there exists a maximizer $\omega^\lambda$ and this maximizer satisfies (1)-(5) in Theorem \ref{51} if $\lambda$ is sufficiently large.

\end{proof}

\appendix
\section{Minimum Points of $H_2$}
In this appendix we give location of the minimum points of the function $H_2$.

\begin{proposition}\label{80}
Let $D=\{\mathbf{x}\in\mathbb{R}^2\mid |x|<1\}$, $\kappa_1>0,\kappa_2<0$ be two real numbers, and $\gamma=-\frac{\kappa_2}{\kappa_1}$. Denote $M$ the set of all minimum points of $H_2$, where
\begin{equation}\label{61}
H_2(\mathbf{x},\mathbf{y}):=-2\kappa_1\kappa_2G(\mathbf{x},\mathbf{y})+{\kappa_1}^2h(\mathbf{x},\mathbf{x})+{\kappa_2}^2h(\mathbf{y},\mathbf{y}),\,\,\mathbf{x},\mathbf{y}\in D, \mathbf{x}\neq\mathbf{y}.
\end{equation}
Then there exists $p\in(0,1),q\in(-1,0)$ depending only on $\gamma$, such that
\begin{equation}\label{62}
M=\{(P,Q)\in D\times D\mid P=p(cos\theta,\sin\theta), Q=q(cos\theta,sin\theta), \theta\in[0,2\pi)\}.
\end{equation}
If $\gamma=1$, then
\[p=-q=\sqrt{\sqrt{5}-2}.\]

\end{proposition}
\begin{proof}
First, it is easy to see that
\begin{equation}
  \lim_{|\mathbf{x}-\mathbf{y}|\rightarrow 0}H_2(\mathbf{x},\mathbf{y})=+\infty, \lim_{\mathbf{x}\rightarrow\partial D\text{ or }\mathbf{y}\rightarrow\partial D}H_2(\mathbf{x},\mathbf{y})=+\infty,
\end{equation}
so $H_2$ attains its minimum in $D\times D\setminus\{(\mathbf{x},\mathbf{y})\in D\times D\mid \mathbf{x}\neq\mathbf{y}\}$, or equivalently, $M$ is not empty.

For $\mathbf{x},\mathbf{y}\in D, \mathbf{x}\neq\mathbf{y}$,

\begin{equation}\label{63}
\begin{split}
H_2(\mathbf{x},\mathbf{y})&=-2\kappa_1\kappa_2G(\mathbf{x},\mathbf{y})+{\kappa_1}^2h(\mathbf{x},\mathbf{x})+{\kappa_2}^2h(\mathbf{y},\mathbf{y}) \\
&=-2\kappa_1\kappa_2(-\frac{1}{2\pi}\ln\frac{|\mathbf{x}-\mathbf{y}|}{|\mathbf{y}||\mathbf{x}-\mathbf{y}^*|})-\kappa_1^2(\frac{1}{2\pi}\ln|\mathbf{x}||\mathbf{x}-\mathbf{x}^*|)
-\kappa_2^2(\frac{1}{2\pi}\ln|\mathbf{y}||\mathbf{y}-\mathbf{y}^*|)\\
&=\frac{\kappa_1^2}{\pi}\ln\frac{|\mathbf{y}|^{2\gamma}|\mathbf{x}-\mathbf{y}^*|^{2\gamma}}{|\mathbf{x}-\mathbf{y}|^{2\gamma}|\mathbf{x}||\mathbf{y}|^{\gamma^2}|\mathbf{y}-\mathbf{y}^*|^{\gamma^2}},
\end{split}
\end{equation}
where $\mathbf{x}^*={\mathbf{x}}/{|\mathbf{x}|^2}$ and $\mathbf{y}^*={\mathbf{y}}/{|\mathbf{y}|^2}$.

So it suffices to consider the minimum points of $T(\mathbf{x},\mathbf{y})$ defined by

\[T(\mathbf{x},\mathbf{y}):=\frac{|\mathbf{y}|^{2\gamma}|\mathbf{x}-\mathbf{y}^*|^{2\gamma}}{|\mathbf{x}-\mathbf{y}|^{2\gamma}|\mathbf{x}||\mathbf{y}|^{\gamma^2}|\mathbf{y}
-\mathbf{y}^*|^{\gamma^2}}.\]

By using the polar coordinates,
\[T(\mathbf{x},\mathbf{y}):=\frac{(1+\rho_1^2\rho_2^2-2\rho_1\rho_2cos(\theta_1-\theta_2))^\gamma}{(\rho_1^2+\rho_2^2-2\rho_1\rho_2cos(\theta_1-\theta_2))^\gamma(1-\rho_1^2)(1-\rho_2^2)^{\gamma^2}},\]
where $\mathbf{x}=\rho_1(cos\theta_1,sin\theta_1)$ and $\mathbf{y}=\rho_2(cos\theta_2,sin\theta_2)$.

Now we show that if $(\mathbf{x},\mathbf{y})$ is a minimum point of $H_2$, then
\[\theta_1-\theta_2=\pi.\]
In fact, it is not hard to check that for fixed $\rho_1,\rho_2$, $T$ is strictly increasing in $cos(\theta_1-\theta_2)$, so at any minimum point we must have
\[cos(\theta_1-\theta_2)=-1.\]

To finish the proof it suffices to consider the minimization of the following function:
\[R(\rho_1,\rho_2):=\frac{(1+\rho_1\rho_2)^{2\gamma}}{(\rho_1+\rho_2)^{2\gamma}(1-\rho_1^2)(1-\rho_2^2)^{\gamma^2}}\]
for $\rho_1,\rho_2\in(0,1)$.

\textbf{Case 1:} $\gamma=1$.
In this simple case, $R(\rho_1,\rho_2)$ becomes
\[R(\rho_1,\rho_2):=\frac{(1+\rho_1\rho_2)^{2}}{(\rho_1+\rho_2)^{2}(1-\rho_1^2)(1-\rho_2^2)},\quad \rho_1,\rho_2\in(0,1).\]
By direct calculation, we obtain
\begin{equation}\label{65}
\partial_{\rho_1}R=0\Leftrightarrow \rho_1\rho_2+\rho_1^2+\rho_1^3\rho_2+2\rho_2^2=1,
\end{equation}
\begin{equation}\label{66}
\partial_{\rho_2}R=0\Leftrightarrow \rho_2\rho_1+\rho_2^2+\rho_2^3\rho_1+2\rho_1^2=1.
\end{equation}
Subtracting the two expressions in \eqref{65} and \eqref{66} we get
\[(1+\rho_1\rho_2)(\rho_1^2-\rho_2^2)=0,\]
which gives $\rho_1=\rho_2.$ Now we can see that $\rho_1$ satisfies
\[\rho_1^4+4\rho_1^2-1=0,\]
so $\rho_1=\rho_2=\sqrt{\sqrt{5}-2}$.

\textbf{Case 2}: $\gamma>0$ is arbitrary. In this case, we show that $R(\rho_1,\rho_2)$ has a unique minimum point for $\rho_1,\rho_2\in(0,1)$. Existence is obvious since we have proved that $M$ is not empty. Now we show the uniqueness. In fact, it suffices to prove that the critical point of $R$ in $(0,1)\times(0,1)$ is unique.

Direct calculation gives
\[\partial_{\rho_1}R=0\Leftrightarrow \rho_1\rho_2+(1+\gamma)\rho_1^2+\gamma\rho_2^2+\rho_1^3\rho_2+(1-\gamma)\rho_1^2\rho_2^2-\gamma=0,\]
\[\partial_{\rho_2}R=0\Leftrightarrow \gamma\rho_1\rho_2+\rho_1^2+\rho_2^2+(\gamma-1)\rho_1^2\rho_2^2+\gamma^2\rho_2^2+\gamma\rho_1\rho_2^3-1=0.\]
For simplicity we write
\[F_1(\rho_1,\rho_2):=\rho_1\rho_2+(1+\gamma)\rho_1^2+\gamma\rho_2^2+\rho_1^3\rho_2+(1-\gamma)\rho_1^2\rho_2^2-\gamma\]
and
\[F_2(\rho_1,\rho_2):=\gamma\rho_1\rho_2+\rho_1^2+\rho_2^2+(\gamma-1)\rho_1^2\rho_2^2+\gamma^2\rho_2^2+\gamma\rho_1\rho_2^3-1.\]
It is not hard to check that $F_1, F_2$ are both strictly monotonic in $\rho_1$ for fixed $\rho_2$ and in $\rho_2$ for fixed $\rho_1$ if $\rho_1,\rho_2\in(0,1)$, which means that the system
$F_1(\rho_1,\rho_2)=0, F_2(\rho_1,\rho_2)=0$ has most one solution. In other words, under the condition $\theta_1-\theta_2=\pi$,  $H_2$ has a unique minimum point, which completes the proof.
\end{proof}

\smallskip

{\bf Acknowledgements:}
{\it D Cao was supported by NNSF of China (grant No. 11331010) and Chinese Academy of Sciences by grant QYZDJ-SSW-SYS021. G Wang was supported by NNSF of China (grant No. 11771469).}

\renewcommand\refname{References}
\renewenvironment{thebibliography}[1]{%
\section*{\refname}
\list{{\arabic{enumi}}}{\def\makelabel##1{\hss{##1}}\topsep=0mm
\parsep=0mm
\partopsep=0mm\itemsep=0mm
\labelsep=1ex\itemindent=0mm
\settowidth\labelwidth{\small[#1]}%
\leftmargin\labelwidth \advance\leftmargin\labelsep
\advance\leftmargin -\itemindent
\usecounter{enumi}}\small
\def\newblock{\ }
\sloppy\clubpenalty4000\widowpenalty4000
\sfcode`\.=1000\relax}{\endlist}
\bibliographystyle{model1b-num-names}

\begin{thebibliography}{32}
\bibitem{AS}
A. Ambrosetti and  M. Struwe,  Existence of steady vortex rings in
an ideal fluid,  \textit{Arch. Ration. Mech. Anal.},  108(1989),
97--109.

\bibitem{A}
V. I. Arnold, Mathematical methods of classical mechanics, \textit{Graduate Texts in Mathematics, Vol. 60.} Springer, New York, 1978.


\bibitem{AK}
V. I. Arnold and B.A. Khesin, Topological methods in hydrodynamics,
\textit{Applied Mathematical Sciences, Vol. 125.} Springer, New
York, 1998.



\bibitem{Ba}
T. V. Badiani, Existence of steady symmetric vortex pairs on a planar
domain with an obstacle, \textit{Math. Proc. Cambridge Philos.
Soc.}, 123(1998), 365--384.


\bibitem{BF}
M. S. Berger and L. E. Fraenkel, Nonlinear desingularization in
certain free-boundary problems, \textit{Comm. Math. Phys.},
77(1980), 149--172.



\bibitem{B1}
G. R. Burton,  Vortex rings in a cylinder and rearrangements,
\textit{J. Differential Equations},  70(1987), 333--348.


\bibitem{B2}

G.R. Burton, Rearrangements of functions, saddle points and
uncountable families of steady configurations for a vortex,
\textit{Acta Math.}, 163(1989), 291--309.
\bibitem{B3}

G.R. Burton, Steady symmetric vortex pairs and rearrangements,
\textit{Proc. R. Soc. Edinb. Sec.A}, 108(1988), 269-290.


\bibitem{B4}
G. R. Burton, Variational problems on classes of rearrangements and multiple configurations for steady vortices, \textit{Ann. Inst. Henri Poincar\'e. Analyse Nonlineare.}, 6(1989), 295-319.
\bibitem{B5}
G.R. Burton, Global nonlinear stability for steady ideal fluid flow in bounded planar domains,
\textit{Arch. Ration. Mech. Anal.}, 176(2005), 149-163.



\bibitem{CLW}
D. Cao, Z. Liu and J. Wei, Regularization of point vortices  for the
Euler equation in dimension two, \textit{Arch. Ration. Mech. Anal.},
212(2014), 179--217.
\bibitem{CPY}
D. Cao, S. Peng and S. Yan, Planar vortex patch problem in incompressible steady flow,\textit{ Adv. Math.}, 270(2015), 263--301.
\bibitem{CW}
D. Cao and G. Wang, Steady double vortex patches with opposite signs in a planar ideal fluid, arXiv:1709.07115.

\bibitem{EM}
A. R. Elcrat and K. G. Miller, Rearrangements in steady vortex flows with circulation, \textit{Proc. Amer. Math. Soc.}, 111(1991), 1051-1055.
\bibitem{EM2}
A. R. Elcrat and K. G. Miller, Rearrangements in steady multiple vortex flows,
 \textit{Comm. Partial Differential Equations}, 20(1994),no.9-10, 1481--1490.


\bibitem{FB}
L. E. Fraenkel and M. S. Berger, A global theory of steady vortex rings in an ideal fluid, \textit{Acta Math.},
132(1974), 13-51.













\bibitem{LYY} G. Li, S. Yan and J. Yang, An elliptic problem related to planar vortex pairs,
\textit{ SIAM J. Math. Anal.}, 36 (2005), 1444-1460.




\bibitem{MPu}
C. Marchioro and M. Pulvirenti, Mathematical theory of incompressible noviscous fluids, Springer-Verlag, 1994.


\bibitem{Ni}
W. M. Ni,  On the existence of global vortex rings,
 \textit{J. Anal. Math.},  37(1980),   208--247.
\bibitem{N}
J. Norbury, Steady planar vortex pairs in an ideal fluid,
 \textit{Comm. Pure Appl. Math.}, 28(1975), 679--700.
 \bibitem{SV}
D. Smets and  J. Van Schaftingen, Desingulariation of vortices for
the Euler equation,
 \textit{ Arch. Ration. Mech. Anal.},  198(2010),   869--925.



 \bibitem{T}
B. Turkington, On steady vortex flow in two dimensions. I,
 \textit{Comm. Partial
Differential Equations}, 8(1983), 999--1030.

\bibitem{T2}
B. Turkington, On steady vortex flow in two dimensions. II,
 \textit{Comm. Partial
Differential Equations}, 8(1983), 1031--1071.

\end{thebibliography}

\end{document}